\documentclass[11pt]{article}

\usepackage{amsfonts, amssymb, amsmath, verbatim, amsthm, enumerate, mathtools}

\usepackage{graphicx}
\usepackage{framed}
\usepackage{tikz}
\usepackage{enumitem}
\usepackage{bbm}
\usepackage{ascmac,amscd,eucal}

\usepackage{hyperref}
\usepackage{setspace}

\usepackage{geometry}
\newcommand{\brac}[1]{\left(#1\right)}
\newcommand{\scalar}[1]{\left \langle #1 \right \rangle}

\renewcommand{\S}{\mathbb{S}}
\newcommand{\eps}{\varepsilon}

\newcommand{\Id}{\mathrm{Id}}
\newcommand{\tr}{\mathrm{tr}}

\newcommand{\bsig}{\overline{\Sigma}}
\DeclareMathOperator{\diag}{diag}
\DeclareMathOperator{\Span}{span}

\newtheorem{theorem}{Theorem}[section]
\newtheorem{lemma}[theorem]{Lemma}
\newtheorem{proposition}[theorem]{Proposition}

\newtheorem{corollary}[theorem]{Corollary}
\theoremstyle{definition}

\theoremstyle{definition}
\newtheorem{remark}[theorem]{Remark}

\newcommand{\vertiii}[1]{{\left\vert\kern-0.25ex\left\vert\kern-0.25ex\left\vert #1 
\right\vert\kern-0.25ex\right\vert\kern-0.25ex\right\vert}}
\newcommand{\R}{{\mathbb R}}

\numberwithin{equation}{section}

\def\1{\textbf{\rm 1}}

\def\XXint#1#2#3{{\setbox0=\hbox{$#1{#2#3}{\int}$}
\vcenter{\hbox{$#2#3$}}\kern-.5\wd0}}

\makeatletter
\@namedef{subjclassname@2020}{\textup{2020} Mathematics Subject Classification}
\makeatother

\title{Gaussian inequalities with different covariance matrices -- the Heterovariant Gaussian Brunn-Minkowski inequality}

\date{ }

\author{
  Emanuel Milman\thanks{
    Department of Mathematics, Technion-Israel Institute of Technology, Haifa 32000, Israel. 
    Email: emilman@tx.technion.ac.il
  }
  \and
  Shohei Nakamura\thanks{
    School of Mathematics, The Watson Building, University of Birmingham, Edgbaston, Birmingham, B15 2TT, England. 
    Email: s.nakamura@bham.ac.uk
  }
  \and
  Hiroshi Tsuji\thanks{
    Department of Mathematics, Institute of Science Tokyo, 2-12-1 Ookayama, Meguro-ku, Tokyo 152-8551, Japan
    Email: tsujihiroshi@math.sci.isct.ac.jp
  }
}

\begin{document}

\maketitle

\begin{abstract}
We establish the following novel heterovariant Gaussian Brunn--Minkowski inequality (involving general covariance matrices $\Sigma_1,\Sigma_2,\Gamma$) for origin-symmetric convex sets $K,L \subset \R^n$ and $p,\lambda \in (0,1)$:
\[
    \gamma_{\Sigma_1}(K)^{1-p} \gamma_{\Sigma_2}(L)^p
    \le
    \left(  \left( \frac{1-p}{1-\lambda} \right)^{1-p}  \left( \frac{p}{\lambda} \right)^p \right)^{n}
    \gamma_{\Gamma}((1-\lambda) K + \lambda L) ,
\]
whenever $\Gamma \leq \Sigma_1 \sharp_p \Sigma_2$, the $p$-geometric mean of $\Sigma_1,\Sigma_2$. 
This immediately implies both the dimensional Gaussian Brunn--Minkowski inequality of Eskenazis--Moschidis (by using $\Sigma_1 = \Sigma_2 = \Gamma$ and optimizing over $p$) and the strong Gaussian (B)-theorem of Cordero-Erausquin--Fradelizi--Maurey (by using $K=L$, $p=\lambda$ and diagonal covariances), fusing them into a single unified statement. In fact, the inequality applies to convex sets whose appropriate Gaussian barycenter is at the origin. Furthermore, it reveals a new \emph{bilinear} version of the strong Gaussian (B)-theorem (involving \emph{two} convex sets); in particular, for all $a,b > 0$ and $\lambda \in [0,1]$:
\[
\gamma(a K)^{1-\lambda} \gamma(b L)^{\lambda} \leq \gamma\brac{a^{1-\lambda} b^\lambda \brac{(1-\lambda) K + \lambda L}} .
\]
A more general heterovariant version of the (bilinear) (B)-theorem replaces dilation by general linear transformations. 
A different version yields a heterovariant form of the recently established Gaussian conjugate Rogers--Shephard inequality. 
By averaging these Gaussian inequalities with respect to the covariance matrices, we obtain dimensional Brunn--Minkowski and conjugate Rogers--Shephard inequalities for Gaussian mixtures. The proofs are based on the Gaussian Forward-Reverse Brascamp--Lieb inequality established in our previous work.

\end{abstract}

\begingroup
\renewcommand{\thefootnote}{}
\footnotetext{  


    \textit{}
  The research leading to these results is part of a project that has received funding from the European Research Council (ERC) under the European Union's Horizon 2020 research and innovation programme (grant agreement No 101001677), and is also supported by JSPS Kakenhi grant number 24KJ0030 (Tsuji).
}
\endgroup

\section{Introduction}

Many fundamental geometric inequalities in convex geometry compare the Lebesgue measures of sets obtained through operations such as Minkowski addition and intersection. In modern convex geometry, it is also natural to replace Lebesgue measure by other measures and ask whether analogous geometric inequalities continue to hold. This point of view has led to the study of a wide variety of geometric inequalities for different classes of measures.
Among them, the standard Gaussian measure in $\R^n$,
$
d\gamma=(2\pi)^{-\frac n2} e^{-\frac12 |x|^2}\, dx ,
$
is one of the most fundamental and important examples, and numerous Gaussian inequalities have been considered. 

A well-known consequence of the Pr\'ekopa--Leindler inequality is the following dimension-free Gaussian Brunn-Minkowski inequality: 
\begin{equation} \label{e:PL}
\gamma(K)^{1-\lambda} \gamma(L)^{\lambda} \le \gamma((1-\lambda) K + \lambda L) 
\end{equation}
for all $\lambda \in [0,1]$ and convex sets $K,L \subset \R^n$ (in fact, convexity is not needed). Here, for $a,b \in \R$ and $K,L \subset \R^n$, we write
$
aK + bL := \{ a x + b y\, :\, x \in K, y \in L\}
$
for their Minkowski combination. By using appropriate translations and a scaling-limit, the Gaussian inequality recovers the classical dimension-free Brunn--Minkowski inequality, where $\gamma$ above is replaced by the Lebesgue measure \cite{AGA-Book-I}. 

In recent years, it has been realized that a richer family of inequalities can hold when restricting to origin-symmetric convex sets $K=-K$ and $L=-L$, or more generally, to appropriately centered convex sets. A representative example is the dimensional Gaussian Brunn--Minkowski inequality, conjectured by Gardner--Zvavitch \cite{GZ} and finally established by Eskenazis--Moschidis \cite{EM} (following the approach of \cite{EMilmanKolesnikovAJM, EMilmanKolesnikovMemoirs, KolesnikovLivshyts}) for origin-symmetric convex sets $K,L \subset \R^n$:
\begin{equation}\label{e:GZ}
(1-\lambda) \gamma(K)^{1/n} + \lambda \gamma(L)^{1/n}
\le
\gamma((1-\lambda) K + \lambda L)^{1/n}. 
\end{equation}
See also \cite{AishwaraRotem,AishwaraLi} for additional extensions. 
Recently, Han--Zhu \cite{HZ} extended (\ref{e:GZ}) to convex sets  $K,L\subset \R^n$  whose Gaussian barycenters coincide: $\frac{1}{\gamma(K)} \int_K x \, d\gamma(x) = \frac{1}{\gamma(L)} \int_L x\, d\gamma(x)$. 
Throughout this work, we will not attempt to treat the most general case when Gaussian barycenters coincide, but rather restrict to the case when the appropriate Gaussian barycenters are both at the origin, e.g.~ $\int_K x \, d\gamma(x) = \int_L x\, d\gamma(x) =0$.
Again, by using appropriate translations and a scaling-limit, the Gaussian inequality (\ref{e:GZ}) recovers the classical dimensional Brunn--Minkowski inequality, where $\gamma$ is replaced by the Lebesgue measure.

The dimensional Gaussian Brunn--Minkowski inequality \eqref{e:GZ} can be understood as a comparison between the ``sizes'' of the convex sets $K$ and $L$ and that of their Minkowski combination $(1-\lambda)K+\lambda L$, where all three sets are measured with respect to Gaussian measures having the same covariance structure. Indeed, for an $n\times n$ positive definite symmetric matrix $\Sigma$, let the Gaussian measure on $\R^n$ with covariance matrix $\Sigma$ be defined by
$$
d\gamma_\Sigma := ({\rm det} (2\pi \Sigma))^{-\frac 12} e^{-\frac12 \langle \Sigma^{-1}x, x \rangle}\, dx.
$$
Then \eqref{e:GZ} may equivalently be written as
\begin{equation}\label{e:GZ2}
(1-\lambda) \gamma_\Sigma(K)^{1/n} + \lambda \gamma_\Sigma(L)^{1/n}
\le
\gamma_\Sigma((1-\lambda) K + \lambda L)^{1/n}
\end{equation}
for all $\lambda \in [0,1]$ and convex sets $K,L\subset \R^n$ with $\int_K x \, d\gamma_\Sigma(x) = \int_L x\, d\gamma_\Sigma(x)=0$.
This naturally raises the question of whether an analogous geometric inequality remains valid when the ``sizes'' of the three convex sets $K$, $L$, and $(1-\lambda)K+\lambda L$ are measured with respect to Gaussian measures having different covariance matrices. More precisely, given $n\times n$ positive definite symmetric matrices $\Sigma_1,\Sigma_2,\Gamma$, convex sets $K,L\subset\R^n$, and $\lambda\in(0,1)$, one may ask under which conditions
$$
(1-\lambda) \gamma_{\Sigma_1}(K)^{1/n} + \lambda \gamma_{\Sigma_2}(L)^{1/n}
\le
\gamma_{\Gamma}((1-\lambda) K + \lambda L)^{1/n}
$$
holds.
This question is not specific to the dimensional Gaussian Brunn--Minkowski inequality, but arises naturally for additional Gaussian inequalities as well, to be discussed below.

There are several reasons for studying Gaussian inequalities involving different covariance matrices. The first is to understand how the validity of such inequalities depends on the covariance matrix. A striking example is provided by the strong (B)-theorem of Cordero-Erausquin--Fradelizi--Maurey \cite{C-EFM}. It states that
\begin{equation} \label{e:StrongB}
\R^n \ni (s_1, \dots, s_n) \mapsto \gamma({\rm diag}(e^{s_1}, \dots, e^{s_n}) K) \text{ is log-concave}
\end{equation}
for any origin-symmetric convex set $K = -K$ in $\R^n$. 
Cordero-Erausquin--Rotem \cite{C-ER} investigated the strong (B)-property for general Gaussian measures and showed that the corresponding property can fail in general, even when the covariance matrix is arbitrarily close to that of the standard Gaussian. Thus, changing the covariance matrix can affect whether a Gaussian inequality remains valid. It is therefore natural to ask which Gaussian properties continue to hold and which matrix conditions guarantee them.

A second reason pertains to Gaussian mixtures. Given a probability measure $\nu$ on $(0, \infty)^n$, let us consider the Gaussian mixture $\gamma_\nu$ on $\R^n$ given by
\begin{equation} \label{e:GaussMix}
\gamma_\nu
=
\int_{(0, \infty)^n}
\gamma_{\Delta(t)} \,d\nu(t)  \quad,\quad \Delta(t) := {\rm diag}(t_1,\ldots,t_n) . 
\end{equation}
This construction gives a class of distributions broader than the class of individual Gaussian measures.
For instance, it is known that measures corresponding to symmetric $p$-stable random variables with $p \in (0,2]$ are Gaussian mixtures. In particular, the Cauchy distribution is a Gaussian mixture. 
Eskenazis--Nayar--Tkocz \cite{ENT} studied entropy and moment inequalities for Gaussian mixtures, as well as extensions of the (B)-theorem and the Gaussian correlation inequality. Eskenazis--Gavalakis \cite{EG} also studied geometric aspects of more general multivariate Gaussian mixtures. Inequalities relating Gaussian measures with different covariances can serve as pointwise ingredients for establishing inequalities for Gaussian mixtures. More precisely, a typical strategy is to first establish an inequality for each choice of covariance parameters, and then integrate it with respect to the mixing measure, using an appropriate functional inequality in the covariance variables.

One last reason, which we did not a priori expect but realized along our investigation, is that using different covariances actually reveals a common ancestral inequality, which underlies and implies the seemingly unrelated inequalities discussed above. We commence by formulating it below. 

\subsection{The Heterovariant Gaussian Brunn--Minkowski Inequality}

For $n \in \mathbb{N}$, let $\mathcal{P}_n$ be the set of all $n \times n$ positive definite symmetric matrices.  For $A \in \mathcal{P}_n$ and a linear subspace $E \subseteq \R^n$, let us denote $A|_E:= P_E A P_E^*$ where $P_E:\R^n \to E$ is the orthogonal projection onto $E$.  When $E = \{0\}$, we use the convention that $\det A|_E = 1$. The Lebesgue measure of a measurable $\Omega \subset E$ is denoted by $|\Omega|$. 

Our first main theorem reads as follows. 

\begin{theorem}[Heterovariant Gaussian Brunn--Minkowski Inequality] \label{t:Unify}
    Let $\Sigma_1, \Sigma_2, \Gamma \in \mathcal{P}_n$ and $p, \lambda \in (0,1)$. 
    Then for any convex sets $K,L \subset \R^n$ satisfying 
     \begin{equation}\label{e:Centered}
    \int_{K} x \, d\gamma_{\Sigma_1}(x) = \int_{L} x \, d\gamma_{\Sigma_2}(x) = 0, 
    \end{equation}
    it holds that
    \begin{equation}\label{e:FullHGBM}
\gamma_{\Sigma_1}(K)^{1-p} \gamma_{\Sigma_2}(L)^p \le C(p,\lambda,\boldsymbol{\Sigma},\Gamma) \gamma_{\Gamma} ( (1-\lambda) K+ \lambda L ), 
\end{equation}
where 
\begin{equation}\label{e:FullHGBMCoeff}
C(p,\lambda,\boldsymbol{\Sigma},\Gamma) 
:= 
\max_{E\subseteq\R^n}
 \left[
   \left(\frac{1-p}{1-\lambda}\right)^{1-p}
   \left(\frac{p}{\lambda}\right)^{p}
 \right]^{\dim E}
 \left( 
 \frac{\det \Gamma|_E}{(\det \Sigma_1|_E)^{1-p}(\det \Sigma_2|_E)^p}
 \right)^{1/2}.
\end{equation}
\end{theorem}

\begin{remark} \label{rem:balls}
The constant $C(p,\lambda,\boldsymbol{\Sigma},\Gamma)$ is easily seen to be best-possible for every choice of $p,\lambda\in (0,1)$ and $\Sigma_1,\Sigma_2,\Gamma \in \mathcal{P}_n$. Indeed, given a linear subspace $E \subseteq \R^n$, we may assume that $E \neq \{0\}$ since using $K=L=\R^n$ shows that the optimal constant must be $\geq 1$. Now select $K = \frac{1-p}{1-\lambda} B_E(\eps) \times E^{\perp}$ and $L = \frac{p}{\lambda} B_E(\eps) \times E^{\perp}$ where $B_E(\eps)$ denotes the centered Euclidean ball of radius $\eps > 0$ in $E$, so that $(1-\lambda) K + \lambda L = B_E(\eps) \times E^{\perp}$. Note that $\gamma_{\Sigma}(\Omega \times E^{\perp}) = \gamma_{\Sigma|_E}(\Omega)$ for any measurable $\Omega \subset E$ and $\Sigma \in \mathcal{P}_n$, and that $\gamma_{\Sigma|_E}(\eps \Omega) = \eps^{\dim E} (|\Omega| / \sqrt{\det(2 \pi \Sigma|_E)} + o(\eps))$. Taking the limit as $\eps \downarrow 0$ and maximizing over all subspaces $E$, the optimality of $C(p,\lambda,\boldsymbol{\Sigma},\Gamma)$ easily follows. 
\end{remark}

There is a particular choice of $\Gamma$ which makes the constant $C(p,\lambda,\boldsymbol{\Sigma},\Gamma)$ especially attractive. 
Given $A, B \in \mathcal{P}_n$ and $p\in [0,1]$, the $p$-geometric mean of $A$ and $B$ is defined by 
$$
A \sharp_p B := A^\frac12 ( A^{-\frac12} B A^{-\frac12})^p A^\frac12. 
$$
Note that $A \sharp_0 B = A$, $A \sharp_1 B = B$ and $A \sharp_p A = A$. It is well-known that there exists a natural Riemannian metric on $\mathcal{P}_n$ (given by $g_A(T_1,T_2) = \tr(A^{-1} T_1 A^{-1} T_2)$), invariant under linear changes of coordinates, such that $[0,1] \ni p \mapsto A \sharp_p B$ is a constant speed geodesic from $A$ to $B$ in this metric, see e.g.~\cite[Chapter 6]{BhatiaBook}. Clearly $\det(A \sharp_p B) = \det(A)^{1-p} \det(B)^p$ and $A^{-1} \sharp_p B^{-1} = (A \sharp_p B)^{-1}$, but it is not true in general that e.g.~$A^2 \sharp_p B^2 = (A \sharp_p B)^2$ nor that $(A \sharp_p B)|_E = (A|_E) \sharp_p (B|_E)$.

\begin{theorem}[Geometric-Mean Heterovariant Gaussian Brunn--Minkowski Inequality] \label{t:HGBM} \label{t:BMCov}
    Let $\Sigma_1, \Sigma_2, \Gamma \in \mathcal{P}_n$ and $p,\lambda \in (0,1)$. 
    Suppose that 
    \begin{equation}\label{e:BMCovCond}
    \Gamma \le \Sigma_1 \sharp_p \Sigma_2.  
    \end{equation}
    Then for any convex sets $K, L \subset \R^n$ satisfying 
    the centering assumption \eqref{e:Centered}, it holds that 
    \begin{equation}\label{e:MultiGBM}
    \gamma_{\Sigma_1}(K)^{1-p} \gamma_{\Sigma_2}(L)^p
    \le
    \left[  \left( \frac{1-p}{1-\lambda} \right)^{1-p}  \left( \frac{p}{\lambda} \right)^p \right]^{n}
    \gamma_{\Gamma}((1-\lambda) K + \lambda L). 
    \end{equation}
\end{theorem}
\begin{proof}
$\Gamma \le \Sigma_1 \sharp_p \Sigma_2$ implies $\Gamma|_E \le \Sigma_1|_E \sharp_p \Sigma_2|_E$ 
by Ando's transformer inequality (see e.g.~\cite{KuboAndo}) and hence 
$\det \Gamma|_E \le (\det \Sigma_1|_E)^{1-p}(\det \Sigma_2|_E)^p$
for all linear subspaces $E \subseteq \R^n$. In addition, $\left(\frac{1-p}{1-\lambda}\right)^{1-p}
   \left(\frac{p}{\lambda}\right)^{p} \geq 1$ by AM-GM,    and therefore $C(p,\lambda,\boldsymbol{\Sigma},\Gamma) \leq 
 \left[
   \left(\frac{1-p}{1-\lambda}\right)^{1-p}
   \left(\frac{p}{\lambda}\right)^{p}
 \right]^{n}$. 
\end{proof}

As in Remark \ref{rem:balls}, the constant on the right-hand side is best-possible. 

Let us now verify that Theorem \ref{t:HGBM} contains both the dimensional Gaussian Brunn--Minkowski inequality \eqref{e:GZ2} and the strong (B)-theorem \eqref{e:StrongB}, fusing them into a single unified formulation,  and revealing several striking new inequalities. 

\subsection{The dimensional Gaussian Brunn--Minkowski inequality}

We first consider the case when $\Sigma_1 =\Sigma_2 = \Gamma$ all coincide with a common $\Sigma \in \mathcal{P}_n$. Let $K,L$ satisfy $\int_K x \, d\gamma_{\Sigma}(x) = \int_L x \, d\gamma_{\Sigma}(x) = 0$. In that case, letting $p = \lambda$, we immediately recover the dimension-free Gaussian Brunn--Minkowski inequality (\ref{e:PL}) (with general covariance $\Sigma$). 

But what happens when we let $p$ vary? In that case, we observe that \eqref{e:MultiGBM} for all $p \in (0,1)$ is actually equivalent to the dimensional Gaussian Brunn--Minkowski \eqref{e:GZ2}. Indeed, \eqref{e:GZ2} and the AM-GM inequality imply:
\begin{align*}
\gamma_{\Sigma}((1-\lambda) K + \lambda L)^{1/n}
&\ge
    (1-\lambda) \gamma_{\Sigma}(K)^{1/n} + \lambda \gamma_{\Sigma}(L)^{1/n}
\\
&\ge
\left[\left( \frac{1-p}{1-\lambda} \right)^{1-p}  \left( \frac{p}{\lambda} \right)^p  \right]^{-1} \gamma_{\Sigma}(K)^{(1-p)/n} \gamma_{\Sigma}(L)^{p/n} ,
\end{align*} 
yielding \eqref{e:MultiGBM} with $\Sigma_1 = \Sigma_2 = \Gamma = \Sigma$. 
Conversely, starting with \eqref{e:MultiGBM} and using that $\Sigma \sharp_p \Sigma = \Sigma$ for every $p \in (0,1)$, we obtain:
\[
\left[\left( \frac{1-p}{1-\lambda} \right)^{1-p}  \left( \frac{p}{\lambda} \right)^p  \right]^{-n}\gamma_{\Sigma}(K)^{1-p} \gamma_{\Sigma}(L)^p
\le \gamma_\Sigma((1-\lambda)K+\lambda L) \quad \forall p \in (0,1) . 
\]
Maximizing the left-hand side over $p \in (0,1)$, it is straightforward to check that the dimensional Gaussian Brunn--Minkowski inequality \eqref{e:GZ2} is obtained, fully recovering the Eskenazis--Moschidis theorem (as well as the Han--Zhu extension when the corresponding barycenters are at the origin):

\begin{corollary} \label{cor:dimensionalGBM}
The dimensional Gaussian Brunn--Minkowski inequality \eqref{e:GZ2} holds for all convex $K,L \subset \R^n$ satisfying $\int_K x \, d\gamma_{\Sigma}(x) = \int_L x \, d\gamma_{\Sigma}(x) = 0$. 
\end{corollary}

\begin{remark}
A similar argument was used by Livshyts--Marsiglietti--Nayar--Zvavitch \cite{LMNZ}, who showed that for any ideals $A, B \subset \R^n$ and any unconditional product measure $\nu$ on $\R^n$ with a non-increasing density, one has  
$$
    \nu(A)^{1-p} \nu(B)^p
    \le
    \left[  \left( \frac{1-p}{1-\lambda} \right)^{1-p} \left( \frac{p}{\lambda} \right)^p \right]^{n}
    \nu((1-\lambda) A + \lambda B) ,
$$
thereby deducing a dimensional Brunn--Minkowski inequality for $\nu$. 
\end{remark}

Using the more general Theorem \ref{t:Unify}, we can also obtain the following heterovariant version of the dimensional Gaussian Brunn--Minkowski inequality, whose proof is deferred to Section \ref{sec:3}. 
For $a,b\ge0$, $\lambda\in (0,1)$ and $-\infty \le p <\infty$, we denote 
$$
M_p^\lambda (a,b)
:=
\begin{cases}
  ((1-\lambda) a^p + \lambda b^p)^{1/p}, &  p\in \R, \\
  \min(a, b), & p=-\infty. 
\end{cases}
$$
\begin{proposition}\label{p:HeteroGBMChar}
    Let $\Sigma_1, \Sigma_2, \Gamma \in \mathcal{P}_n$ and $\lambda \in (0,1)$. Then the following statements are equivalent:
    \begin{enumerate}[label=(\roman*)]
        \item \label{it:char-i} For any convex $K,L\subset \R^n$ satisfying the centering assumption \eqref{e:Centered}, it holds that 
        $$
        (1-\lambda) \gamma_{\Sigma_1}(K)^{1/n} + \lambda \gamma_{\Sigma_2}(L)^{1/n}
        \le
        \gamma_{\Gamma}((1-\lambda) K + \lambda L)^{1/n}. 
        $$
         \item  \label{it:char-ii} For any linear subspace $E \subseteq \R^n$ with ${\rm dim}\, E=d \geq 1$, it holds that 
        \[         {\rm det}\, \Gamma|_E \le M_{-\frac1{2(n-d)}}^\lambda ({\rm det}\, \Sigma_1|_E, {\rm det}\, \Sigma_2|_E). 
        \]         Of course, when $d=n$, $-\frac{1}{2(n-d)}$ is interpreted as $-\infty$. 
    \end{enumerate}
\end{proposition}

\subsection{The (B)-theorem and its extensions}

For $s=(s_1, \dots, s_n) \in \R^n$, let $e^s := (e^{s_1}, \dots, e^{s_n}) \in \R^n$ and $\Delta(s) := {\rm diag}(s_1, \dots, s_n) \in \R^{n \times n}$. 

\subsubsection{The strong (B)-theorem}

The (B)-theorem, originally conjectured by Banaszczyk, was popularized by Lata\l a \cite{Lat} and finally established, in a strengthened form, by Cordero-Erausquin--Fradelizi--Maurey \cite{C-EFM}. Recall that the strong (B)-theorem asserts that $\R^n \ni s \mapsto \gamma(\Delta(e^s) K)$ is log-concave for every origin-symmetric convex set $K=-K$ in $\R^n$. In other words, we have for all $p \in [0,1]$,
\begin{equation}\label{e:FormStrongBThm}
\gamma(\Delta(e^s)K)^{1-p} \gamma(\Delta(e^t)K)^p
    \le
    \gamma( \Delta(e^{(1-p)s + p t}) K). 
\end{equation}
In particular, specializing $\Delta(e^s)$ and $\Delta(e^t)$ to multiples of the identity, the (B)-theorem asserts that for all $a,b > 0$,
\begin{equation} \label{e:B-Thm}
\gamma(aK)^{1-p} \gamma(bK)^p \le \gamma(a^{1-p}b^p K). 
\end{equation}

Now note that $\gamma_{\Sigma}(K) = \gamma(\Sigma^{-\frac{1}{2}} K)$ and $\int_K x \, d\gamma_{\Sigma}(x) = \Sigma^{\frac{1}{2}} \int_{\Sigma^{-\frac{1}{2}} K} x \, d\gamma(x)$, two facts we shall freely use below. Consequently, specializing Theorem \ref{t:HGBM} to the case when $K=L$ is origin-symmetric, $\lambda=p$, $\Sigma_1 = \Delta(e^{-2s})$, $\Sigma_2 = \Delta(e^{-2t})$ and $\Gamma = \Sigma_1 \sharp_p \Sigma_2 = \Delta(e^{-2((1-p) s + p t)})$, we immediately recover the strong (B)-theorem \eqref{e:FormStrongBThm}. 
In fact, we see that symmetry is not essential:
\begin{corollary}
Let $p \in [0,1]$ and $s,t\in \R^n$. Then the strong (B)-theorem \eqref{e:FormStrongBThm} holds for all convex $K \subset \R^n$ satisfying
\[
\int_{\Delta(e^s) K} x \, d\gamma(x) = \int_{\Delta(e^t) K} x \, d\gamma(x) = 0 . 
\]
\end{corollary}

\subsubsection{The heterovariant (B)-theorem}

Over the past few decades, extending the strong (B)-theorem to arbitrary even log-concave measures has attracted considerable attention because of its connection with the log-Brunn--Minkowski conjecture put forth by B\"{o}r\"{o}czky--Lutwak--Yang--Zhang \cite{BLYZ}. 
Indeed, it was shown by Saroglou \cite{SarMath} that the log-Brunn--Minkowski conjecture is equivalent to the conjecture that 
\[ \R^n \ni s \mapsto \gamma_{\Sigma}(\Delta(e^s) [-1, 1]^n)
\] is log-concave for any $\Sigma \in \mathcal{P}_n$ and $n \in \mathbb{N}$. 

Cordero-Erausquin--Rotem \cite{C-ER} observed that the log-concavity of $\gamma_{\Sigma}(\Delta(e^s) K)$ cannot hold for arbitrary origin-symmetric convex sets $K \subset \R^n$ and $\Sigma \in \mathcal{P}_n$. More precisely, they showed that for every $\eps > 0$, there exist an origin-symmetric convex set $K_\eps \subset \R^n$ and $\Sigma_{\eps} \in \mathcal{P}_n$ within $\eps$ distance to ${\rm Id}_n$ such that $\gamma_{\Sigma_\eps}(\Delta(e^s) K_\eps)$ is not log-concave in $s$. This shows that \eqref{e:FormStrongBThm} is not preserved under arbitrarily small perturbations of the covariance matrix. 

The following new corollary of Theorem \ref{t:HGBM} provides one possible extension of the strong (B)-theorem to general (possibly non-diagonal, non-commuting) non-singular matrices. 

\begin{corollary}[Heterovariant (B)-theorem] \label{c:HeteroBThm}
Let $Q,P$ denote non-singular $n$ by $n$ matrices, and let $p \in [0,1]$. 
Then for any convex set $K \subset \R^n$ satisfying
\[
\int_{Q K} x \, d\gamma(x) = \int_{P K} x \, d\gamma(x) = 0 ,
\]
it holds that
\[
\gamma(Q K)^{1-p} \gamma(P K)^p \leq \gamma(((Q^t Q) \sharp_p (P^t P))^{\frac{1}{2}} K) . 
\]
In particular, if $A,B \in \mathcal{P}_n$ and
\[
\int_{A^{\frac{1}{2}} K} x \, d\gamma(x) = \int_{B^{\frac{1}{2}} K} x \, d\gamma(x) = 0 ,
\]
then
\[
\gamma(A^{\frac{1}{2}} K)^{1-p} \gamma(B^{\frac{1}{2}} K)^p \leq \gamma((A \sharp_p B)^{\frac{1}{2}} K) . 
\]
\end{corollary}
\begin{proof}
Let us first establish the ``in particular" statement. Apply Theorem \ref{t:HGBM} with $K=L$, $\lambda=p$, $\Sigma_1 = A^{-1}$, $\Sigma_2 = B^{-1}$ and $\Gamma = \Sigma_1 \sharp_p \Sigma_2 = (A \sharp_p B)^{-1}$.  Recalling that $\gamma_{\Sigma}(K) = \gamma(\Sigma^{-\frac{1}{2}} K)$ and $\int_K x \, d\gamma_{\Sigma}(x) = \Sigma^{\frac{1}{2}} \int_{\Sigma^{-\frac{1}{2}} K} x \, d\gamma(x)$, the ``in particular" part follows.

For the general case, write $Q = U_Q A^{\frac{1}{2}}$ and  $P = U_P B^{\frac{1}{2}}$  for orthogonal $U_Q,U_P$ and $A,B \in \mathcal{P}_n$, so that $A = Q^t Q$ and $B = P^t P$. Since the Gaussian measure $\gamma$ is rotation invariant, we have $\gamma(Q K) = \gamma(A^{\frac{1}{2}} K)$ and $\gamma(P K) = \gamma(B^{\frac{1}{2}} K)$, reducing the general case to the particular one. 
\end{proof}

Using $Q = A^{1/2} = \Delta(e^s)$ and $P = B^{1/2} =\Delta(e^t)$, this recovers the strong (B)-theorem \eqref{e:FormStrongBThm}.

\subsection{The bilinear (B)-theorem}

Instead of using $K=L$ as in the previous subsection, a new bilinear form of the (B)-theorem and its extensions arises if we keep $K$ and $L$ separate (but still use $\lambda = p$). Repeating verbatim the argument used above, Theorem \ref{t:HGBM} implies:

\begin{corollary}[Bilinear Heterovariant (B)-theorem] \label{cor:BHB-Thm}
Let $Q,P$ denote two non-singular $n$ by $n$ matrices, and let $\lambda \in [0,1]$. Then for any convex sets $K, L \subset \R^n$ satisfying
    $$
    \int_{Q K} x \, d\gamma(x) = \int_{P L} x \, d\gamma(x) = 0, 
    $$
    it holds that 
    \begin{equation}\label{e:BiBThm}
    \gamma(Q K)^{1-\lambda} \gamma(P L)^\lambda
    \le
    \gamma \brac{((Q^t Q) \sharp_\lambda (P^t P))^{\frac{1}{2}}\brac{ (1-\lambda) K + \lambda L}}. 
    \end{equation}
\end{corollary}

Taking $K=L$ recovers the heterovariant (B)-Theorem (Corollary \ref{c:HeteroBThm}). On the other hand, using $P = Q = \Id_n$ recovers the dimension-free Gaussian Brunn-Minkowski inequality (\ref{e:PL}). Even the simplest non-trivial case which mixes these two extremes together, namely when $Q = a \, \Id_n$ and $P = b \, \Id_n$, is already new and interesting:

\begin{corollary}[Bilinear (B)-theorem] \label{cor:BB-Thm}
For all $a,b > 0$ and convex sets $K, L \subset \R^n$ satisfying
\[
\int_{a K} x \, d\gamma(x) = \int_{b L} x \, d\gamma(x) = 0, 
\]
it holds for all $\lambda \in [0,1]$ that 
\[
\gamma(a K)^{1-\lambda} \gamma(b L)^\lambda \le  
    \gamma(a^{1-\lambda} b^\lambda ((1-\lambda) K + \lambda L)). 
\]
\end{corollary}

Taking $K=L$, a common origin-symmetric convex body, recovers the usual (B)-Theorem \eqref{e:B-Thm}. The following equivalent formulation is perhaps more insightful:

\begin{corollary}[Bilinear (B)-theorem again - an improved dimension-free Gaussian Brunn--Minkowski Inequality] \label{cor:BBE-Thm}
For any convex sets $K, L \subset \R^n$ satisfying
\[
\int_{K} x \, d\gamma(x) = \int_{L} x \, d\gamma(x) = 0, 
\]
it holds for all $\lambda \in [0,1]$ that 
\begin{equation} \label{eq:bilinear-B-thm}
\gamma(K)^{1-\lambda} \gamma(L)^\lambda \le \inf_{r > 0} 
    \gamma((1-\lambda) r^{-\lambda} K + \lambda r^{1-\lambda} L). 
\end{equation}
\end{corollary}
\begin{proof}
Applying Corollary \ref{cor:BB-Thm} with $K/a$, $L/b$, we obtain
\[
\gamma(K)^{1-\lambda} \gamma(L)^\lambda \le
    \gamma(a^{1-\lambda} b^\lambda ( (1-\lambda) K/a + \lambda L/b)). 
\]
Denoting $r =a/b$, the assertion follows.  
\end{proof}

Using $r = 1$, we see that Corollary \ref{cor:BBE-Thm} improves the usual dimension-free Gaussian Brunn--Minkowski inequality \eqref{e:PL}. 
There is a particularly interesting feature of this formulation: if we could move the infimum over $r > 0$ inside the Gaussian measure, we would obtain
\[
\gamma(K)^{1-\lambda} \gamma(L)^\lambda \le     \gamma\brac{\bigcap_{r > 0} \brac{(1-\lambda) r^{-\lambda} K + \lambda r^{1-\lambda} L}} 
\]
for all origin-symmetric convex $K,L \subset \R^n$. 
Denoting the convex set on the right-hand side by $C$, its support function $h_C : \S^{n-1} \rightarrow \R_+$ satisfies
\[
h_C(\theta) \leq \inf_{r > 0}  (1-\lambda) r^{-\lambda} h_K(\theta) + \lambda r^{1-\lambda} h_L(\theta) = h_K^{1-\lambda}(\theta) h_L^\lambda(\theta), \;\;\; \forall \theta \in \mathbb{S}^{n-1} .
\]
In fact, $C$ is the set-wise largest convex body satisfying this property for all $\theta \in \S^{n-1}$, referred to as the Alexandrov body $A[h_K^{1-\lambda} h_L^\lambda]$. Consequently, this would confirm the following equivalent formulation of the log-Brunn--Minkowski conjecture (see Saroglou \cite{SarMath}):
\begin{equation} \label{eq:GlogBM}
\gamma(K)^{1-\lambda} \gamma(L)^\lambda \le \gamma\brac{A[h_K^{1-\lambda} h_L^\lambda]} . 
\end{equation}

\begin{remark}[Bilinear (B)-theorem implies dimensional Gaussian Brunn--Minkowski]
The above argument shows that the (conjectured) Gaussian log-Brunn--Minkowski inequality (\ref{eq:GlogBM}) logically implies the bilinear (B)-theorem (\ref{eq:bilinear-B-thm}) for all origin-symmetric convex sets $K,L \subset \R^n$ (in fact, for fixed $\lambda \in (0,1)$). It was shown by Livshyts--Marsiglietti--Nayar--Zvavitch \cite{LMNZ} that the (conjectured) Gaussian log-Brunn--Minkowski inequality (for all $\lambda \in (0,1)$) also logically implies the dimensional Gaussian Brunn--Minkowski inequality (\ref{e:GZ}) for all origin-symmetric convex sets $K,L \subset \R^n$ (for all $\lambda \in (0,1)$). An interesting observation is that the weaker bilinear (B)-theorem (for all $\lambda \in (0,1)$) actually suffices to logically imply the dimensional Gaussian Brunn--Minkowski inequality (for all $\lambda \in (0,1)$). In fact, these logical implications do not depend on the Gaussian measure $\gamma$, and as observed in \cite[Proposition 1]{LMNZ} for the ``log-BM $\Rightarrow$ dimensional BM" implication, equally hold for any measure $\mu$ having density $f_\mu$ which is radially non-increasing in all directions -- for example, any $\mu$ with even log-concave density. This implies that $\mu(t A) \geq t^n \mu(A)$ for all $t \in [0,1]$ and measurable $A \subset \R^n$ (see \cite[equation (15)]{LMNZ}), which is the only property we require below. 
Indeed, assume that
\begin{equation} \label{eq:mu-bilinear-B-thm}
\mu(K)^{1-p} \mu(L)^p \le \inf_{r > 0} 
    \mu((1-p) r^{-p} K + p r^{1-p} L) 
\end{equation}
for all origin-symmetric $K,L \subset \R^n$ and $p \in (0,1)$. Given $\lambda,p \in (0,1)$, denote 
\[
r = \frac{\lambda (1-p)}{p (1-\lambda)} ~,~ t = \brac{\frac{1-\lambda}{1-p}}^{1-p} \brac{\frac{\lambda}{p}}^p   , 
\]
and note that $t \in (0,1]$ by AM-GM. These choices ensure
\[
t (1-p)r^{-p}=1-\lambda ~,~ tp r^{1-p}=\lambda .
\]
Applying (\ref{eq:mu-bilinear-B-thm}) to $tK,tL$ and using $\mu(t A) \geq t^n \mu(A)$ yields
\[
\mu((1-\lambda)K+\lambda L) \geq \gamma(tK)^{1-p}\gamma(tL)^p \geq t^n\gamma(K)^{1-p}\gamma(L)^p.
\]
Recalling the definition of $t$ and optimizing over $p \in (0,1)$ as in Corollary \ref{cor:dimensionalGBM}, it follows that
\[
\mu((1-\lambda)K+\lambda L)^{1/n} \geq (1-\lambda) \mu(K)^{1/n} + \lambda \mu(L)^{1/n} . 
\]
\end{remark}

\subsection{Monotonicity}

It is worthwhile noting that there is a bit more flexibility in the choice of covariance matrices for which Theorem \ref{t:HGBM} (and its numerous consequences) remain valid. This is thanks to the following (possibly known) lemma. 

\begin{lemma} \label{lem:monotonicity}
Let $\Sigma_- , \Sigma_+ \in \mathcal{P}_n$ with $\Sigma_- \leq \Sigma_+$, and let $K \subset \R^n$ be convex. Assume that 
\[
\int_K x \, d\gamma_{\Sigma_*}(x) = 0 
\]
for some $* \in \{-,+\}$. Then
\[
\gamma_{\Sigma_-}(K) \geq \gamma_{\Sigma_+}(K) . 
\]
\end{lemma}

When $K$ is origin-symmetric this is due to Anderson \cite{And}, but we could not find a reference for the more general case above -- we provide a proof in Section \ref{sec:prelim}. It follows from using $* = -$ in conjunction with  Theorem \ref{t:HGBM} that the centering assumption \eqref{e:Centered} implies
\[
\gamma_{\Gamma_1}(K)^{1-p} \gamma_{\Gamma_2}(L)^p
    \le
    \left(  \left( \frac{1-p}{1-\lambda} \right)^{1-p}  \left( \frac{p}{\lambda} \right)^p \right)^{n}
    \gamma_{\Gamma}((1-\lambda) K + \lambda L). 
\]
 for all $\Gamma_i \geq \Sigma_i$ and $\Gamma \leq \Sigma_1 \sharp_p \Sigma_2$. Note that Theorem \ref{t:HGBM} was formulated with the flexibility of choosing any $\Gamma \leq \Sigma_1 \sharp_p \Sigma_2$, since there is no assumption on the $\gamma_{\Gamma}$ nor $\gamma_{\Sigma_1 \sharp_p \Sigma_2}$ barycenters of $(1-\lambda) K + \lambda L$. The same flexibility on either side of the inequality applies to Theorem \ref{t:Unify} and all of the corollaries of Theorem \ref{t:HGBM} obtained above.

\subsection{The heterovariant Gaussian conjugate Rogers--Shephard inequality}

We next turn to the Gaussian Conjugate Rogers--Shephard Inequality (GCRSI) with different covariances. In our previous work \cite{MNT}, we showed that for convex sets $K,L\subset\R^n$ satisfying $\int_K x\,d\gamma=\int_L x\,d\gamma=0$, it holds that 
\begin{equation}\label{e:GCRSI}
\gamma(K) \gamma(L) \le \gamma(K \cap L) \gamma( K +L).
\end{equation}
For origin-symmetric convex sets, this inequality was originally conjectured by Tehranchi \cite{Teh}. The following result extends the covariance comparison to all four Gaussian terms.

\begin{theorem}[Heterovariant GCRSI]\label{t:TehCov}
    Let $\Sigma_1, \Sigma_2, \Gamma_1, \Gamma_2 \in \mathcal{P}_n$. 
    Suppose that 
    \begin{equation}\label{e:TehCovCond}
    \Sigma_1 \le \Gamma_2, \;\;\;
    \Sigma_2 \le \Gamma_2, \;\;\;
    \Gamma_1^{-1} + \Gamma_2^{-1} \le \Sigma_1^{-1} + \Sigma_2^{-1}, 
    \;\;\;
    {\rm det}\, \Sigma_1 {\rm det}\, \Sigma_2 \ge {\rm det}\, \Gamma_1 {\rm det}\, \Gamma_2.  
    \end{equation}
    Then for any convex sets $K, L \subset \R^n$ satisfying 
    $$
    \int_{K} x \, d\gamma_{\Sigma_1}(x) = \int_{L} x \, d\gamma_{\Sigma_2}(x) = 0, 
    $$
    it holds that 
    $$
    \gamma_{\Sigma_1}(K) \gamma_{\Sigma_2}(L)
    \le
    \gamma_{\Gamma_1}(K \cap L) \gamma_{\Gamma_2}(K + L). 
    $$
\end{theorem}

Taking $\Sigma_1=\Sigma_2 = \Gamma_1 = \Gamma_2 = {\rm Id}_n$ immediately recovers \eqref{e:GCRSI}. 

\begin{remark}
    The assumption \eqref{e:TehCovCond} forces $\Gamma_1 \le \Sigma_1, \Sigma_2$ and $\Gamma_1^{-1} + \Gamma_2^{-1} = \Sigma_1^{-1} + \Sigma_2^{-1}$. 
    To see this, let $X_i := \Gamma_2^{-\frac12} \Sigma_i \Gamma_2^{-\frac12}$ for $i=1,2$ and $Y:= \Gamma_2^{-\frac12} \Gamma_1 \Gamma_2^{-\frac12}$. Then \eqref{e:TehCovCond} may be reformulated as 
    \begin{equation}\label{e:TehCovCond'}
    X_1, X_2 \le {\rm Id}_n, \;\;\; Y^{-1} + {\rm Id}_n \le X_1^{-1} + X_2^{-1}, 
    \;\;\;
    {\rm det}\, X_1 {\rm det}\, X_2 \ge {\rm det}\, Y. 
    \end{equation}
    Then the second and first inequalities yield that 
    \begin{align*}
        ({\rm det}\, Y)^{-1}
        &\le
        {\rm det}\, (X_1^{-1} + X_2^{-1} - {\rm Id}_n)
        \\
        &
        =
        {\rm det}\, (X_1^{-1}X_2^{-1} - (X_1^{-1} - {\rm Id}_n)(X_2^{-1} - {\rm Id}_n) )
        \le
        {\rm det}\, (X_1^{-1}X_2^{-1}).
    \end{align*}
    Hence by the third inequality in \eqref{e:TehCovCond'}, we have 
    $$
    {\rm det}\,  Y = {\rm det}\, X_1 {\rm det}\, X_2. 
    $$
    This fact implies that 
    $$
    Y^{-1} + {\rm Id}_n = X_1^{-1} + X_2^{-1}, 
    $$
    which yields $\Gamma_1^{-1} + \Gamma_2^{-1} = \Sigma_1^{-1} + \Sigma_2^{-1}$. 
    Moreover combining this with the first inequality in \eqref{e:TehCovCond'}, we deduce $Y \le X_1, X_2$. This yields $\Gamma_1 \le \Sigma_1, \Sigma_2$. 
\end{remark}

In our previous work \cite{MNT}, we considered not only the GCRSI \eqref{e:GCRSI}, but also the conjugate Milman--Pajor inequality and related Gaussian inequalities. It is natural to expect analogous results for these Gaussian inequalities with different covariance matrices. However, we do not pursue this direction here and leave it to the interested reader.

\subsection{Gaussian mixture inequalities}
Finally, we obtain various geometric inequalities for Gaussian mixtures derived from the preceding Gaussian inequalities. 
Given a  probability measure $\nu$ on $(0, \infty)^n$, recall the definition of a Gaussian mixture $\gamma_\nu$ from (\ref{e:GaussMix}). 
Several geometric and information-theoretic properties of Gaussian mixtures, including extensions of the (B)-theorem, were investigated in \cite{ENT, EG}.

Our first result  for Gaussian mixtures is the dimensional Brunn--Minkowski inequality, which is a consequence of Theorem \ref{t:BMCov}. 
\begin{corollary}[Dimensional Brunn--Minkowski inequality for Gaussian mixtures] \label{c:GaussMixDimBM}
    Let $\lambda \in (0,1)$, and let $\nu$ be a probability measure on $(0, \infty)^n$ such that its density $f_\nu$ satisfies that $\R^n \ni u \mapsto f_\nu(e^u)$ is log-concave. 
    Then for any convex sets $K, L \subset \R^n$ such that 
    $$
    \int_{K} x \, d\gamma_{\Delta(z)}(x) = \int_{L} x \, d\gamma_{\Delta(z)}(x) = 0, \;\;\; \forall\, z \in {\rm supp}(\nu),  
    $$
    it holds that 
    \begin{equation}\label{e:GaussMixBM}
    (1-\lambda) \gamma_{\nu}(K)^{1/n} + \lambda \gamma_{\nu}(L)^{1/n}
    \le
    \gamma_{\nu}((1-\lambda) K + \lambda L)^{1/n}. 
    \end{equation}
\end{corollary}

The Brunn--Minkowski inequality for general even log-concave measures is one of the central problems in contemporary convex geometry. 
The two-dimensional case has already been established by Livshyts--Marsiglietti--Nayar--Zvavitch \cite{LMNZ}, as a consequence of the two-dimensional log-Brunn--Minkowski conjecture for the Lebesgue proved by B\"{o}r\"{o}czky--Lutwak--Yang--Zhang \cite{BLYZ} and its extension to general even log-concave measures by Saroglou \cite{SarMath}. 
However, the general-dimensional case remains open, and the current best result, allowing a weaker exponent, is due to Yang \cite{Yang} (following prior contributions in \cite{LivshytsTAMS, EGT}). 
See \cite{Yang} and the references therein for more details on the log-Brunn--Minkowski conjecture. 
Note that our Gaussian mixture \eqref{e:GaussMix} need not be log-concave in general, and thus Corollary \ref{c:GaussMixDimBM} lies outside this regime. 
For example, it is known from \cite{ENT} that, for $p \in (0, 1]$, the distribution on $\R^n$ with density proportional to $e^{-\sum_{i=1}^n |x_i|^p}$  satisfies the assumption in Corollary \ref{c:GaussMixDimBM}, although it is not log-concave for $0<p<1$.  
The Brunn--Minkowski inequality for Gaussian mixtures was also studied by Eskenazis--Nayar--Tkocz \cite{ENT}, who established \eqref{e:GaussMixBM} for origin-symmetric convex sets $K,L$ that are homothetic to each other. 
Corollary \ref{c:GaussMixDimBM} extends their result to arbitrary pairs of convex sets $K,L\subset \R^n$ satisfying the prescribed barycenter condition. 
For Brunn--Minkowski inequalities for other measures beyond the log-concave setting, see \cite{C-ERAnn, LMNZ, Mars}.

Our second result for Gaussian mixtures is the conjugate Rogers--Shephard inequality (CRSI). 
Before stating it, let us introduce the notion of log-supermodularity. 
A nonnegative function $f$ on $(0, \infty)^n$ is said to be log-supermodular if
$$
f(x) f(y) \le f(x \vee y) f(x \wedge y), \;\;\; \forall x, y \in (0, \infty)^n, 
$$
where for $x=(x_1, \dots, x_n)$ and $y=(y_1, \dots, y_n)$ in $\R^n$, the vectors $x \vee y, x \wedge y \in \R^n$ are defined by
$$
x \vee y := ( \max(x_1, y_1), \dots, \max(x_n, y_n)) ~,~ 
x \wedge y := ( \min(x_1, y_1), \dots, \min(x_n, y_n)).
$$

\begin{corollary}[CRSI for Gaussian mixtures] \label{c:TehGaussMix}
    Let $\nu$ be a probability measure on $(0, \infty)^n$ with log-supermodular density $f_\nu$. 
    Then for any convex sets $K, L \subset \R^n$ such that 
    $$
    \int_{K} x \, d\gamma_{\Delta(s)}(x) = \int_{L} x \, d\gamma_{\Delta(s)}(x) = 0, \;\;\; \forall\, s \in {\rm supp}(\nu),  
    $$
    it holds that 
    $$
    \gamma_{\nu}(K) \gamma_{\nu}(L)
    \le
    \gamma_{\nu}(K \cap L) \gamma_{\nu}(K + L).  
    $$
\end{corollary}

In particular, since $\gamma_\nu$ is a probability measure, we also obtain the following correlation inequality for Gaussian mixtures under the same assumption:
\begin{equation} \label{e:GCI}
\gamma_{\nu}(K) \gamma_{\nu}(L)
    \le
    \gamma_{\nu}(K \cap L). 
\end{equation}
The original correlation inequality for the Gaussian measure $\gamma$ and arbitrary origin-symmetric convex sets was established by Royen \cite{Roy}. A second proof was obtained in \cite{M}, and the symmetry assumption was replaced by a Gaussian barycenter condition in \cite{NT}. 
The correlation inequality \eqref{e:GCI} for origin-symmetric convex $K,L$ was previously established for $\gamma_\nu$ which is a product of one-dimensional Gaussian mixtures (corresponding to a product measure $\nu$, whose density is trivially log-supermodular) by Eskenazis--Nayar--Tkocz \cite{ENT}. 

As mentioned before, it is natural to expect analogous conjugate Milman--Pajor and related inequalities, as developed in \cite{MNT}, for Gaussian mixtures. However, we do not pursue these questions here and leave them to the interested reader.

The rest of this paper is organized as follows. 
In Section \ref{sec:prelim}, we recall the forward-reverse Brascamp--Lieb inequality established in our previous work \cite{MNT} and present several inequalities for positive definite matrices that will be used in the subsequent sections. 
In Section \ref{sec:3}, we prove the dimensional Brunn--Minkowski inequality for Gaussian measures with different covariance matrices and for Gaussian mixtures. 
Section \ref{sec:4} is devoted to the conjugate Rogers--Shephard inequality for Gaussian measures with different covariance matrices and Gaussian mixtures.

\section*{Acknowledgments}
We thank Nathael Gozlan for suggesting that we consider the conjugate Rogers--Shephard inequality for Gaussian mixtures.

\section*{AI use statement}

The results of this work were obtained prior to OpenAI's October 6, 2026 unveiling, while the paper was being written. 
In the age of AI domination and proliferation, we are proud to declare that no AI was used in the derivation of any of the mathematical proofs or ideas, with the exception of testing the validity of (the linear algebra) Lemmas \ref{l:Diagonalize} and \ref{l:TehCovMat}. The paper was written entirely by the authors, but AI was used to prune the paper for typos and insert references. The authors take full responsibility for everything written.

\section{Preliminaries} \label{sec:prelim}

\subsection{The Gaussian Forward-Reverse Brascamp--Lieb inequality}

The main tool used throughout this paper is the Gaussian Forward-Reverse Brascamp--Lieb inequality established in our previous work \cite{MNT}. 
For further background on this inequality and its applications, see \cite{MNT} and the references therein. For $A \in \mathcal{P}_n$, we denote the centered Gaussian $g_A(x) := \exp(-\frac{1}{2} \scalar{Ax,x})$. 

\begin{theorem}[{\cite[Theorem 1.10]{MNT}}]\label{t:FRBLMNT}
    Let $E= \oplus_{i=1}^I \R^{n_i}$. Let $L_j:E \to \R^{m_j}$ be linear surjective maps, $j=1, \dots, J$, let $(c_i)_{i=1}^I \in (0, \infty)^I$ and $(d_j)_{j=1}^J \in (0, \infty)^J$, and let $\mathcal{Q}$ be a symmetric matrix on $E$. 
    Let $\Sigma_i \in \mathcal{P}_{n_i}$ for  $i=1 \dots, I$, and let $\Gamma_j$ be an $m_j \times m_j$ positive semi-definite matrix for $j=1 \dots, J$. 
    Then for all log-concave $f_i :\R^{n_i} \to \R_+$ so that $\int_{\R^{n_i}} x f_i(x)\, d\gamma_{\Sigma_i}(x)=0$, $i=1, \dots, I$ and for all log-concave $h_j :\R^{m_j} \to \R_+$, $j=1, \dots, J$, it holds that 
    \begin{align*}
        &\prod_{i=1}^I f_i(x_i)^{c_i} \le e^{-\frac12 \langle \mathcal{Q}x, x \rangle} \prod_{j=1}^J h_j(L_jx)^{d_j}, 
        \;\;\; \forall x=(x_1, \dots, x_I) \in \bigoplus_{i=1}^I \R^{n_i}=E
        \\
        &\Rightarrow \;
        \prod_{i=1}^I \left( \int_{\R^{n_i}} f_i\, d\gamma_{\Sigma_i} \right)^{c_i}
        \le
        {\rm FR}_{LC}^{(\mathcal{G})}
        \prod_{j=1}^J \left( \int_{\R^{m_j}} h_j\, d\gamma_{\Gamma_j} \right)^{d_j}. 
    \end{align*}
    Here ${\rm FR}_{LC}^{(\mathcal{G})} \in (0, \infty]$ is the best constant when $f_i$ and $h_j$ in the above implication are restricted to centered Gaussians $g_{A_i}$ and $g_{B_j}$, i.e., 
    $$
    {\rm FR}_{LC}^{(\mathcal{G})}
    :=
    \sup \left\{
    \frac{\prod_{j=1}^J {\rm det}\, ({\rm Id}_{m_j} + \Gamma_j B_j)^{d_j/2}}{\prod_{i=1}^I {\rm det}\, ({\rm Id}_{n_i} + \Sigma_i A_i)^{c_i/2}}
    \, :\, 
    {\rm diag}(c_1A_1, \dots, c_IA_I)
    \ge 
    \mathcal{Q} + \sum_{j=1}^J d_j L_j^* B_j L_j
    \right\}, 
    $$
    where the supremum is taken over all $A_i \in \mathcal{P}_{n_i}$ and $B_j \in \mathcal{P}_{m_j}$, and $L_j^*:\R^{m_j} \to E$ is the adjoint operator of $L_j$, $j=1, \dots, J$.  
\end{theorem}

\subsection{Several inequalities on matrices}

Let us collect several inequalities that will be useful for simplifying the arguments in the following sections. 

\begin{lemma}[{\cite[Corollary 5.4]{MNT}}]\label{l:GCIGauss}
    For any $\mathcal{A}, \mathcal{B}, \mathcal{C} \in \mathcal{P}_n$, one has 
    $$
    {\rm det}\, (\mathcal{C})\, {\rm det}\, (\mathcal{C} + \mathcal{A} + \mathcal{B}) 
    \le 
    {\rm det}\, (\mathcal{C} + \mathcal{A} )\,
    {\rm det}\, (\mathcal{C} + \mathcal{B}). 
    $$
\end{lemma}

\begin{lemma} \label{lem:det-monotonic}
For any $A,B, \mathcal{A}, \mathcal{B} \in \mathcal{P}_n$, if $A \leq \mathcal{A}$ and $B \leq \mathcal{B}$, then
\[
\det(\Id_n + A B) \leq \det(\Id_n + \mathcal{A} \mathcal{B}) . 
\]
\end{lemma}
\begin{proof}
It is well-known that for any $n$ by $n$ matrices $X,Y$ we have $\det(\Id_n + X Y) = \det(\Id_n + Y X)$. Therefore
\[
\det(\Id_n + A B) = \det(\Id_n + B^{1/2} A B^{1/2}) \leq \det(\Id_n + B^{1/2} \mathcal{A} B^{1/2}) = \det(\Id_n + \mathcal{A} B) ,
\]
where we used that $B^{1/2} A B^{1/2} \leq B^{1/2} \mathcal{A} B^{1/2}$. Repeating the same argument a second time, the assertion follows. 
\end{proof}

\subsection{Monotonicity}

We conclude this section by providing a proof of Lemma \ref{lem:monotonicity}. 

\begin{proof}[Proof of Lemma \ref{lem:monotonicity}]
The simplest way is to invoke the following centered Gaussian Correlation Inequality from the second and third authors' work \cite[Theorem 1.5]{NT}, extending the approach from \cite{M}. It states that for any convex sets $K, L \subset \R^n$ and $\Sigma_0, \Sigma_1, \Sigma_2 \in \mathcal{P}_n$  satisfying $\Sigma_0 \le \Sigma_1, \Sigma_2$ and 
$$
\int_K x\, d\gamma_{\Sigma_1}(x) = \int_L x\, d\gamma_{\Sigma_2}(x) =0,   
$$
it holds that 
\begin{equation} \label{eq:centered-GCI}
\gamma_{\Sigma_0} (K \cap L) \ge \gamma_{\Sigma_1}(K) \gamma_{\Sigma_2}(L). 
\end{equation}

When $\int_K x \, d\gamma_{\Sigma_+}(x) = 0$, letting  $L=\R^n$, $\Sigma_0 = \Sigma_-$ and $\Sigma_1 = \Sigma_2 = \Sigma_+$, we immediately obtain $\gamma_{\Sigma_-}(K) \geq \gamma_{\Sigma_+}(K) $. 

Next let us show the assertion under $\int_K x \, d\gamma_{\Sigma_-}(x) = 0$. 
To see this, note that in general, for any $\Sigma \in \mathcal{P}_n$ and convex $K \subset \R^n$ with $\int_K x \, d\gamma_\Sigma(x) =0$, it holds that 
\begin{equation}\label{e:TranslateCentered}
\sup_{a \in \R^n} \gamma_\Sigma(K-a) = \gamma_\Sigma(K). 
\end{equation}
Indeed, by the Pr\'{e}kopa--Leindler inequality, we see that $\Phi(a):=\gamma_\Sigma(K-a)$ is log-concave in $a \in \R^n$ and 
$$
\nabla \log \Phi(0) = \Sigma^{-1} \int_K x \, \frac{d\gamma_\Sigma(x)}{\gamma_{\Sigma}(K)} =0 
$$
by assumption. Thus $\Phi(a)$ is maximized at $a=0$, yielding \eqref{e:TranslateCentered}. 

Now let us take $b \in \R^n$ such that $\int_{K-b} x \, d\gamma_{\Sigma_+}(x)=0$. Then by \eqref{e:TranslateCentered}, $\int_K x \, d\gamma_{\Sigma_-}(x) = 0$ and the fact that we have already confirmed above, we have 
\begin{align*}
    \gamma_{\Sigma_-}(K)
    \ge
    \gamma_{\Sigma_-}(K-b)
    \ge 
    \gamma_{\Sigma_+}(K-b)
    \ge
    \gamma_{\Sigma_+}(K). 
\end{align*}
Our proof is complete.
\end{proof}

\section{Proofs of Theorem \ref{t:Unify}, Proposition \ref{p:HeteroGBMChar} and Corollary \ref{c:GaussMixDimBM}} \label{sec:3}

In this section we establish Theorem \ref{t:Unify}, Proposition \ref{p:HeteroGBMChar} and Corollary \ref{c:GaussMixDimBM}.
We first require the following lemma. 

\begin{lemma}\label{l:Diagonalize}
Let $p \in (0,1)$ and $Y,Y_1,Y_2\in \mathcal{P}_n$. Then 
\[
\sup_{X_1,X_2 \in \mathcal{P}_n}
\frac{
 \det\!\left(\Id_n+Y\bigl((1-p)X_1^{-1}+p X_2^{-1}\bigr)^{-1}\right)
}{\det(\Id_n+Y_1X_1)^{1-p}\det(\Id_n+Y_2X_2)^p}
=
\sup_{X \in \mathcal{P}_n}
\frac{
 \det\!\left(\Id_n+Y X\right)
}{\det(\Id_n+Y_1X)^{1-p}\det(\Id_n+Y_2 X)^p}
.
\]
\end{lemma}

\begin{proof}
    We simultaneously diagonalize $X_1,X_2$ by congruence.
Namely, choose $U\in O(n)$ such that
\[
U^{*}X_1^{-1/2}X_2X_1^{-1/2}U
=D:=\diag(\kappa_1,\ldots,\kappa_n) ,
\]
and set
$
T:=X_1^{1/2}U.
$
Then we have
\[
X_1=TT^{*},\qquad X_2=TDT^{*},\qquad 
\bigl((1-p)X_1^{-1}+p X_2^{-1}\bigr)^{-1}
=T\bigl((1-p )\Id_n+p D^{-1}\bigr)^{-1}T^{*}.
\]
Note that
\[
\bigl((1-p)\Id_n+p D^{-1}\bigr)^{-1}
=\diag\!\left(\frac{1}{1-p+p/\kappa_1},\ldots,
              \frac{1}{1-p+p/\kappa_n}\right)
=:\diag(r_1,\ldots,r_n).
\]
These identities imply
\begin{align*}
&\frac{
 \det\!\left(\Id_n+Y\bigl((1-p)X_1^{-1}+p X_2^{-1}\bigr)^{-1}\right)
}{\det(\Id_n+Y_1X_1)^{1-p}\det(\Id_n+Y_2X_2)^p}
=
\frac{\det\!\left(\Id_n+\widehat Y\diag(r_1,\ldots,r_n)\right)}
 {\det(\Id_n+\widehat Y_1)^{1-p}\det(\Id_n+\widehat Y_2D)^p},
\end{align*}
where we have used that $\det(\Id_n + A B) = \det(\Id_n + B A)$ for all $n \times n$ matrices $A,B$, and have set
\[
\widehat Y:=T^{*}YT,\qquad
\widehat Y_1:=T^{*}Y_1T,\qquad
\widehat Y_2:=T^{*}Y_2T.
\]
Denoting $E_J:=\Span\{e_j:j\in J\}$ for $J \subseteq [n]$, the principal-minor expansion shows that 
\begin{equation} \label{eq:minor}
\det\!\left(\Id_n+\widehat Y\diag(r_1,\ldots,r_n)\right)
=\sum_{J\subseteq[n]} \prod_{j\in J}r_j \cdot  \det \widehat Y|_{E_J}.
\end{equation}
To estimate $\det \widehat Y|_{E_J}$, we claim that 
\begin{equation}\label{e:Inequ20Oct}
\frac{
\det \widehat Y|_{E_J}}{(\det \widehat Y_1|_{E_J})^{1-p}
(\det \widehat Y_2|_{E_J})^p}
\le 
\sup_{X \in \mathcal{P}_n}
\frac{\det(\Id_n+YX)}
 {\det(\Id_n+Y_1X)^{1-p}\det(\Id_n+Y_2 X)^p}=: \mathfrak{C}.
\end{equation}
To see this, first note that by the change of variables $\overline X=T^{-1}X(T^{*})^{-1}$, the right-hand side is equal to
$$
\sup_{\overline X \in \mathcal{P}_n}
\frac{\det(\Id_n+\widehat Y\overline X)}
 {\det(\Id_n+\widehat Y_1\overline X)^{1-p}
  \det(\Id_n+\widehat Y_2\overline X)^p}.
$$
Hence, taking 
$
\overline X
=R\sum_{j\in J}e_j\otimes e_j
+\frac1R\sum_{j\notin J}e_j\otimes e_j
$ 
and letting $R\to \infty$ gives \eqref{e:Inequ20Oct}. 

Combining (\ref{eq:minor}) with \eqref{e:Inequ20Oct} yields  
\begin{align*}
&\det\!\left(\Id_n+\widehat Y\diag(r_1,\ldots,r_n)\right)
\le
\mathfrak{C} 
\sum_{J\subseteq[n]} \prod_{j\in J}r_j \cdot (\det \widehat Y_1|_{E_J})^{1-p}
(\det \widehat Y_2|_{E_J})^p.
\end{align*}
As for $r_j$, we have 
$
r_j:=\frac{1}{1-p+p/\kappa_j}\le \kappa_j^p
$
by the AM-GM inequality. 
Plugging this above and applying H\"older's inequality, we obtain
\begin{align*}
\det\!\left(\Id_n+\widehat Y\diag(r_1,\ldots,r_n)\right)
&\le
\mathfrak{C} 
\sum_{J\subseteq[n]}
(\det \widehat Y_1|_{E_J})^{1-p}
\left(\prod_{j\in J}\kappa_j \cdot  \det \widehat Y_2|_{E_J}\right)^p\\
&\le
\mathfrak{C} 
\left( 
\sum_{J\subseteq[n]}
\det \widehat Y_1|_{E_J} \right)^{1-p}
\left(
\sum_{ J \subseteq [n]} \prod_{j\in J}\kappa_j \cdot  \det \widehat Y_2|_{E_J}\right)^p \\
 &= 
 \mathfrak{C}  
 \det ( \Id_n + \widehat{Y}_1 )^{1-p} 
 \det ( \Id_n + \widehat{Y}_2 D )^{p} .
\end{align*}
This is precisely the assertion. 
\end{proof}

\begin{proof}[Proof of Theorem \ref{t:Unify}] 
With the choice $(f_1,f_2,h_1)= ( \mathbf{1}_{K}, \mathbf{1}_L, \mathbf{1}_{ (1-\lambda)K+\lambda L })$, we use Theorem \ref{t:FRBLMNT} with the datum $(I,J)=(2,1)$, $(c_1,c_2,d_1)= (1-p,p,1)$, $\mathcal{Q}=0$, and $L_1 = \begin{pmatrix}
    (1-\lambda) \Id_n & \lambda \Id_n
\end{pmatrix} $. Thus the problem reduces to estimating ${\rm FR}_{LC}^{(\mathcal{G})}$, replacing  $(f_1,f_2,h_1)$ by centered Gaussians $(g_{A_1},g_{A_2}, g_{B})$ for $A_1,A_2,B \in \mathcal{P}_n$ satisfying
\begin{equation} \label{eq:satisfying1}
g_{A_1}(x_1)^{1-p}g_{A_2}(x_2)^p \leq g_{B}((1-\lambda) x_1 + \lambda x_2) \;\;\; \forall x_1,x_2 \in \R^n . 
\end{equation}

It is not hard to check using Lagrange multipliers that
\[
\sup_{\substack{x_1,x_2\in\R^n\\
 z=(1-\lambda)x_1+\lambda x_2}}
g_{A_1}(x_1)^{1-p}g_{A_2}(x_2)^p
=g_{B_0}(z),
\quad B_0 :=\left(
\frac{(1-\lambda)^2}{1-p}A_1^{-1}
+\frac{\lambda^2}{p}A_2^{-1}
\right)^{-1}.
\]
Therefore $B \leq B_0$, and so by Lemma \ref{lem:det-monotonic},
\begin{align*}     ({\rm FR}_{LC}^{(\mathcal{G})})^2
    & :=  \sup  \left \{
    \frac{\det\!\left(\Id_n+\Gamma B \right)}{\det(\Id_n+\Sigma_1A_1)^{1-p}\det(\Id_n+\Sigma_2A_2)^p} : \substack{\text{$A_1, A_2, B \in \mathcal{P}_n$ satisfying \eqref{eq:satisfying1}}}
    \right \}
    \\
   & = \sup_{A_1,A_2 \in \mathcal{P}_n}\frac{
\det\!\left(\Id_n+\Gamma
 \left(\frac{(1-\lambda)^2}{1-p}A_1^{-1}
       +\frac{\lambda^2}{p}A_2^{-1}\right)^{-1}\right)
}{\det(\Id_n+\Sigma_1A_1)^{1-p}\det(\Id_n+\Sigma_2A_2)^p}\\
& = \sup_{\overline A_1,\overline A_2 \in \mathcal{P}_n}
\frac{
\det\!\left(\Id_n+\Gamma
  \bigl((1-p)\overline A_1^{-1}+p\overline A_2^{-1}\bigr)^{-1}\right)
}{
\det\!\left(\Id_n+ 
  \bsig_1\overline A_1\right)^{1-p}
\det\!\left(\Id_n+ 
  \bsig_2\overline A_2\right)^p
} ,
\end{align*}
where we used the change of variables 
\begin{equation} \label{eq:bigs}
A_1=\left(\frac{1-\lambda}{1-p}\right)^2\overline A_1 ~,~ \bsig_1:=\left(\frac{1-\lambda}{1-p}\right)^2\Sigma_1  ~,~
A_2=\left(\frac{\lambda}{p}\right)^2\overline A_2 ~,~ \bsig_2:=\left(\frac{\lambda}{p}\right)^2\Sigma_2 .
\end{equation}
By Lemma \ref{l:Diagonalize}, we can reduce to the case that $\overline A_1= \overline A_2$, and hence
$$
({\rm FR}_{LC}^{(\mathcal{G})})^2= 
\sup_{ A \in \mathcal{P}_n}
\frac{
\det\!\left(\Id_n+\Gamma A
\right)
}{
\det\!\left(\Id_n+ 
  \bsig_1 A\right)^{1-p}
\det\!\left(\Id_n+ 
  \bsig_2 A\right)^p
}.
$$
Let $u_1,\ldots,u_n$ be an orthonormal eigenbasis of $A$, with corresponding
eigenvalues $a_1,\ldots,a_n>0$. For $J\subseteq[n]$, set
\[
U_J:=\Span\{u_j:j\in J\}.
\]
The principal-minor expansion gives
\begin{align*}
\det(\Id_n+\Gamma A)
&=\sum_{J\subseteq[n]}
  \prod_{j\in J}a_j \cdot \det \Gamma|_{U_J} \\
&\le
 \left(\max_{E\subseteq\R^n}
  \frac{\det \Gamma|_E}{(\det \bsig_1|_E)^{1-p}(\det \bsig_2|_E)^p}\right)
 \sum_{J\subseteq[n]}\prod_{j\in J}a_j \cdot  (\det \bsig_1|_{U_J})^{1-p}(\det \bsig_2|_{U_J})^p.
\end{align*}
By H\"older's inequality and the principal-minor expansion, we see that 
\begin{align*}
& \sum_{J\subseteq[n]} \prod_{j\in J}a_j \cdot  (\det \bsig_1|_{U_J})^{1-p}(\det \bsig_2|_{U_J})^p \\
 & \leq \brac{\sum_{J\subseteq[n]} \prod_{j\in J}a_j \cdot  \det \bsig_1|_{U_J}}^{1-p} \brac{\sum_{J\subseteq[n]} \prod_{j\in J}a_j \cdot \det \bsig_2|_{U_J}}^{p} \\
 & = \det(\Id_n+\bsig_1A)^{1-p}\det(\Id_n+\bsig_2A)^p ,
 \end{align*}
 and so putting everything together we deduce
 \[
 ({\rm FR}_{LC}^{(\mathcal{G})})^2= 
\sup_{ A \in \mathcal{P}_n}
\frac{
\det\!\left(\Id_n+\Gamma A
\right)
}{
\det\!\left(\Id_n+ 
  \bsig_1 A\right)^{1-p}
\det\!\left(\Id_n+ 
  \bsig_2 A\right)^p
} \leq \max_{E\subseteq\R^n}
  \frac{\det \Gamma|_E}{(\det \bsig_1|_E)^{1-p}(\det \bsig_2|_E)^p} . 
 \]
Recalling the definitions of $ \bsig_i$ from (\ref{eq:bigs}), we conclude that
\[
{\rm FR}_{LC}^{(\mathcal{G})} \leq \max_{E\subseteq\R^n}
 \left[
   \left(\frac{1-p}{1-\lambda}\right)^{1-p}
   \left(\frac{p}{\lambda}\right)^{p}
 \right]^{\dim E}
 \left( 
 \frac{\det \Gamma|_E}{(\det \Sigma_1|_E)^{1-p}(\det \Sigma_2|_E)^p}
 \right)^{1/2}.
\] 
\end{proof}

\begin{proof}[Proof of Proposition \ref{p:HeteroGBMChar}]
    First, let us show that \ref{it:char-ii} implies \ref{it:char-i}. 
    It suffices to show that \ref{it:char-ii} implies that for any $p \in (0,1)$, 
    \begin{equation}\label{e:EstimateOct18}
    C(p,\lambda,\boldsymbol{\Sigma},\Gamma) 
    \le
    \left[ \left( \frac{1-p}{1-\lambda} \right)^{(1-p)}\left( \frac{p}{\lambda} \right)^{p} \right]^n. 
    \end{equation}
    If so, then applying Theorem \ref{t:Unify} and optimizing in $p$ as in Corollary \ref{cor:dimensionalGBM}, \ref{it:char-i} would follows. 

    So fix $p \in (0,1)$, and let $E \subseteq \R^n$ be a linear subspace with ${\rm dim}\, E=d \geq 1$. When $d \le n-1$, it follows from \ref{it:char-ii} and the AM-GM inequality that 
    \begin{align*}
    ({\rm det}\, \Gamma|_E)^{-\frac1{2(n-d)}} 
    &\ge (1-\lambda) ({\rm det}\, \Sigma_1|_E)^{-\frac1{2(n-d)}} + \lambda  ({\rm det}\, \Sigma_2|_E)^{-\frac1{2(n-d)}}
    \\
    &\ge
    \left[
    \left( \frac{1-\lambda}{1-p} \right)^{1-p} \left( \frac{\lambda}{p} \right)^p
    \right]
    ({\rm det}\, \Sigma_1|_E)^{-\frac{1-p}{2(n-d)}}
    ({\rm det}\, \Sigma_2|_E)^{-\frac{p}{2(n-d)}}, 
    \end{align*}
    and thus 
    $$
    \frac{{\rm det}\, \Gamma|_E}{({\rm det}\, \Sigma_1|_E)^{1-p} ({\rm det}\, \Sigma_2|_E)^p}
    \le 
    \left[ \left( \frac{1-p}{1-\lambda} \right)^{(1-p)}\left( \frac{p}{\lambda} \right)^{p} \right]^{2(n-d)}. 
    $$
    Similarly, when $d=n$, \ref{it:char-ii} immediately yields ${\rm det}\, \Gamma \le \min ({\rm det}\, \Sigma_1, {\rm det}\, \Sigma_2)$. 
    Recalling the definition of $C(p,\lambda,\boldsymbol{\Sigma},\Gamma)$ in \eqref{e:FullHGBMCoeff}, the exponent $d$ on the left term combines with the exponent $n-d$ on the right term, yielding the desired exponent $n$ in \eqref{e:EstimateOct18}. 

    Next let us show that \ref{it:char-i} implies \ref{it:char-ii}. 
    To see this, again let $E \subseteq \R^n$ be a linear subspace with ${\rm dim}\, E=d \geq 1$.     Recall from Remark \ref{rem:balls} that $B_E(\eps)$ denotes the centered Euclidean ball in $E$ of radius $\eps>0$, and that $\gamma_{\Sigma}(B_E(\eps) \times E^\perp) = \gamma_{\Sigma|_E}(B_E(\eps))$ for any $\Sigma \in \mathcal{P}_n$. Thus for $\varepsilon >0$ and $p \in (0,1)$, let $K= \frac{1-p}{1-\lambda} B_E(\eps) \times E^\perp$ and $L= \frac{p}{\lambda} B_E(\eps) \times E^\perp$. Since $K,L$ are origin-symmetric and $(1-\lambda) K + \lambda L=  B_E(\eps) \times E^\perp$, we may apply \ref{it:char-i} to see that
    $$
    (1-\lambda) \gamma_{\Sigma_1|_E} \left (\frac{1-p}{1-\lambda} B_E(\eps) \right)^{1/n}
    + \lambda \gamma_{\Sigma_2|_E} \left( \frac{p}{\lambda} B_E(\eps) \right)^{1/n}
    \le
    \gamma_{\Gamma|_E}(B_E(\eps))^{1/n}. 
    $$
    Dividing both sides by $(2\pi)^{-d/(2n)}\varepsilon^{d/n}|B_E|^{1/n}$ and letting $\varepsilon \downarrow 0$ yields 
    $$
    (1-\lambda)^{1-d/n} (1-p)^{d/n} ({\rm det}\, \Sigma_1|_E)^{-1/(2n)} 
    + \lambda^{1-d/n} p^{d/n} ({\rm det}\, \Sigma_2|_E)^{-1/(2n)} 
    \le
    ({\rm det}\, \Gamma|_E)^{-1/(2n)} . 
    $$
    When $d=n$, this clearly implies that $\det \Gamma \leq \min(\det \Sigma_1,\det \Sigma_2)$. Otherwise, optimizing this inequality in $p \in (0,1)$ by selecting 
    \[
    p = \frac{\lambda ({\rm det}\, \Sigma_2|_E)^{-\frac{1}{2(n-d)}}}{(1-\lambda)  ({\rm det}\, \Sigma_1|_E)^{-\frac{1}{2(n-d)}} + \lambda  ({\rm det}\, \Sigma_2|_E)^{-\frac{1}{2(n-d)}}} 
    \]
    easily establishes \ref{it:char-ii}. 
\end{proof}

We conclude this section by establishing Corollary \ref{c:GaussMixDimBM}. 
To this end, we need the following heterovariant version of the Brunn--Minkowski-type inequality for Gaussian mixtures.

\begin{proposition}\label{p:GaussMixBM}
    Let $\lambda, p \in (0,1)$, $s, t \in (0, \infty)^n$ and let $\nu$ be a probability measure on $(0, \infty)^n$ such that its density $f_\nu$ satisfies that $\R^n \ni u \mapsto f_\nu(e^u)$ is log-concave. 
    Then for any convex sets $K, L \subset \R^n$ such that 
    $$
    \int_{K} x \, d\gamma_{\Delta(s)^{-2}\Delta(z)}(x) = \int_{L} x \, d\gamma_{\Delta(t)^{-2}\Delta(z)}(x) = 0, \;\;\; \forall\, z \in {\rm supp}(\nu),  
    $$
    it holds that 
    $$
    \gamma_{\nu}(\Delta(s) K)^{1-p} \gamma_{\nu}(\Delta(t)L)^p
    \le
    \left( \left( \frac{p}{\lambda} \right)^p \left( \frac{1-p}{1-\lambda} \right)^{1-p} \right)^{n}
    \gamma_{\nu}( \Delta(s)^{1-p}\Delta(t)^p ((1-\lambda)  K + \lambda L) ). 
    $$
\end{proposition}

\begin{proof}
For $u, v \in \R^n$, denote
\begin{align*}
 &\Sigma_1 = \Delta(s)^{-2}\Delta(e^u) ~,~ \Sigma_2= \Delta(t)^{-2}\Delta(e^v) ~,~\\
 & \Gamma = \Sigma_1 \sharp_p \Sigma_2 = \Delta(s)^{-2(1-p)} \Delta(t)^{-2p} \Delta(e^{(1-p)u+pv}). 
 \end{align*}
    Then whenever $e^u, e^v \in {\rm supp}(\nu)$, Theorem \ref{t:BMCov} implies that 
    \begin{align*}
    &\gamma_{\Delta(s)^{-2} \Delta(e^u)}(K)^{1-p} \gamma_{\Delta(t)^{-2} \Delta(e^v)}(L)^p
    \\
    &\le
    \left( \left( \frac{p}{\lambda} \right)^p \left( \frac{1-p}{1-\lambda} \right)^{1-p} \right)^{n}
    \gamma_{\Delta(s)^{-2(1-p)} \Delta(t)^{-2p} \Delta(e^{(1-p)u + p v})}((1-\lambda) K + \lambda L). 
    \end{align*}
    This can be rewritten as 
    \begin{align}\label{e:partsOfPL}
    &\gamma_{\Delta(e^u)}(\Delta(s) K)^{1-p} \gamma_{\Delta(e^v)}(\Delta(t) L)^p
    \\
    &\le
    \left( \left( \frac{p}{\lambda} \right)^p \left( \frac{1-p}{1-\lambda} \right)^{1-p} \right)^{n}
    \gamma_{\Delta(e^{(1-p)u + p v})}\brac{\Delta(s)^{1-p} \Delta(t)^p ( (1-\lambda) K + \lambda L) }. \notag
    \end{align}
    For $u,v,w \in \R^n$, define the following functions $f,g,h : \R^n \rightarrow \R_+$:
    $$
    f(u) := \gamma_{\Delta(e^u)}(\Delta(s) K)f_\nu(e^u) e^{\sum_{i=1}^n u_i}, 
    \;\;
    g(v) := \gamma_{\Delta(e^v)}(\Delta(t) L)f_\nu(e^v) e^{\sum_{i=1}^n v_i}, 
    $$
    and 
    $$
    h(w) := \left( \left( \frac{p}{\lambda} \right)^p \left( \frac{1-p}{1-\lambda} \right)^{1-p} \right)^{n} \gamma_{\Delta(e^w)}(\Delta(s)^{1-p} \Delta(t)^p ( (1-\lambda)K + \lambda L) ) f_\nu(e^w) e^{\sum_{i=1}^n w_i}.
    $$
    Then \eqref{e:partsOfPL} and log-concavity of $\R^n \ni u \mapsto f_\nu (e^u)$ implies that 
    $$
    f(u)^{1-p} g(v)^p \le h((1-p)u + pv), \;\;\;u, v \in \R^n. 
    $$
    Indeed, if  $f_\nu(e^u) = 0$ or $f_\nu(e^v)=0$ then the above inequality holds trivially; otherwise $e^u,e^v\in\operatorname{supp}(\nu)$, so \eqref{e:partsOfPL} applies. Thus the Pr\'{e}kopa--Leindler inequality \cite{AGA-Book-I} yields 
    $$
    \left( \int_{\R^n} f(u) \, du \right)^{1-p} \left( \int_{\R^n} g(v) \, dv \right)^{p}
    \le
    \int_{\R^n} h(w) \, dw.   
    $$
    Changing variables $\tilde u = e^u$ (with Jacobian $d\tilde u = e^{\sum_{i=1}^n u_i} du$) and similarly for $v$ and $w$ leads to our desired inequality. 
\end{proof}

\begin{proof}[Proof of Corollary \ref{c:GaussMixDimBM}]
Let $K,L \subset \R^n$ be convex sets satisfying 
    $$
    \int_{K} x \, d\gamma_{\Delta(z)}(x) = \int_{L} x \, d\gamma_{\Delta(z)}(x) = 0, \;\;\; \forall\, z \in {\rm supp}(\nu). 
    $$
    By Proposition \ref{p:GaussMixBM} applied with $\Delta(s)=\Delta(t) = \Id_n$, we have for any $p,\lambda \in (0,1)$, 
    $$
    \gamma_{\nu}(K)^{1-p} \gamma_{\nu}(L)^p
    \le
    \left( \left( \frac{p}{\lambda} \right)^p \left( \frac{1-p}{1-\lambda} \right)^{1-p} \right)^{n}
    \gamma_{\nu}( (1-\lambda)  K + \lambda L) . 
    $$
    The desired assertion now follows by optimizing over $p \in (0,1)$, as in the proof of Corollary \ref{cor:dimensionalGBM} from the Introduction. 
\end{proof}

\section{Proofs of Theorem \ref{t:TehCov} and Corollary \ref{c:TehGaussMix}} \label{sec:4}

As in the previous section, the proof of Theorem \ref{t:TehCov} is reduced, thanks to Theorem \ref{t:FRBLMNT}, to the study of the following problem in linear algebra.  

\begin{lemma}\label{l:TehCovMat}
    Let $\Sigma_1, \Sigma_2, \Gamma_1, \Gamma_2 \in \mathcal{P}_n$. 
    Suppose that 
    \begin{enumerate}[label=(\roman*)]
    \item \label{it:i}  $\Sigma_1 \le \Gamma_2$, $\Sigma_2 \le \Gamma_2$; \smallskip
    \item \label{it:ii}  $\Gamma_1^{-1} + \Gamma_2^{-1} \le \Sigma_1^{-1} + \Sigma_2^{-1}$; and \smallskip
    \item \label{it:iii} ${\rm det}\, \Sigma_1 \, {\rm det}\, \Sigma_2 \ge {\rm det}\, \Gamma_1  \, {\rm det}\, \Gamma_2$.  
    \end{enumerate}
    Then for any $A_1, A_2, C \in \mathcal{P}_n$ and 
    $B = \begin{pmatrix}
        B_1 & B_2 \\ B_2^T & B_4
    \end{pmatrix}\in \mathcal{P}_{2n}$ 
    such that 
    \begin{equation}\label{e:TehCovMatAss}
        \begin{pmatrix}
            A_1  & 0 \\
            0 & A_2
        \end{pmatrix}
        \ge B + 
        \begin{pmatrix}
            C  & C \\
            C & C
        \end{pmatrix}, 
    \end{equation}
    it holds that 
    $$
    {\rm det}\, ({\rm Id}_n + \Gamma_1 \bar B)
    {\rm det}\, ({\rm Id}_n + \Gamma_2 C)
    \le
    {\rm det}\, ({\rm Id}_n + \Sigma_1 A_1)
    {\rm det}\, ({\rm Id}_n + \Sigma_2 A_2),  
    $$
    where $\bar B := B_1 + B_2 + B_2^T + B_4$. 
\end{lemma}
\begin{proof}
    First, note that \eqref{e:TehCovMatAss} and $C>0$ imply      $$
    \bar B \le A_1 + A_2 - 4 C \le A_1 + A_2 - C. 
    $$
    Hence by Lemma \ref{lem:det-monotonic}, 
    \begin{align*}
        &{\rm det}\, ({\rm Id}_n + \Gamma_1 \bar B)
    {\rm det}\, ({\rm Id}_n + \Gamma_2 C)
    \\
    &
    \le
    {\rm det}\, ({\rm Id}_n + \Gamma_1 (A_1 + A_2 - C))
    {\rm det}\, ({\rm Id}_n + \Gamma_2 C)
    \\
    &=
    {\rm det}\, (\Gamma_1\Gamma_2)\, {\rm det}\, (\Gamma_1^{-1} + A_1 + A_2 - C)
    {\rm det}\, (\Gamma_2^{-1} +  C)
    \\
    &\le
    {\rm det}\, (\Gamma_1\Gamma_2)\, {\rm det}\, ( (\Gamma_2^{-1} +  C) + (\Sigma_1^{-1} - \Gamma_2^{-1} + A_1 - C) + (\Sigma_2^{-1} - \Gamma_2^{-1} + A_2 - C) )\,
    {\rm det}\, (\Gamma_2^{-1} +  C), 
    \end{align*}
    where the last inequality follows from \ref{it:ii}.
    Let  
    $$
    \mathcal{A}
    :=
    \Sigma_1^{-1} - \Gamma_2^{-1} + A_1 - C, 
    \;\;
    \mathcal{B}
    :=
    \Sigma_2^{-1} - \Gamma_2^{-1} + A_2 - C, 
    \;\;
    \mathcal{C}
    := \Gamma_2^{-1} +  C. 
    $$
    Note that  $\mathcal{A}, \mathcal{B}, \mathcal{C} \in \mathcal{P}_n$ by \ref{it:i} and by $A_1, A_2 \ge C$ which follows from \eqref{e:TehCovMatAss}. 
    Thus we may apply Lemma \ref{l:GCIGauss} to see that 
    \begin{align*}
        &{\rm det}\, ({\rm Id}_n + \Gamma_1 \overline{B})
    {\rm det}\, ({\rm Id}_n + \Gamma_2 C)
    \\
    &\le
    {\rm det}\, (\Gamma_1)\, {\rm det}\, (\Gamma_2)\, {\rm det}\, ( \Sigma_1^{-1}  + A_1 )\,
    {\rm det}\, ( \Sigma_2^{-1} + A_2 )
    \\
    &=
    \frac{{\rm det}\, (\Gamma_1)\, {\rm det}\, (\Gamma_2)}{{\rm det}\, (\Sigma_1)\, {\rm det}\, (\Sigma_2)} {\rm det}\, ( {\rm Id}_n  + \Sigma_1 A_1 )\,
    {\rm det}\, ( {\rm Id}_n + \Sigma_2 A_2 )
    \\
    &\le
    {\rm det}\, ( {\rm Id}_n  + \Sigma_1 A_1 )\,
    {\rm det}\, ( {\rm Id}_n + \Sigma_2 A_2 ), 
    \end{align*}
    where the last inequality follows from \ref{it:iii}. 
    This concludes the proof.  
\end{proof}

\begin{proof}[Proof of Theorem \ref{t:TehCov}]
    Put $f_1=\mathbf{1}_K$, $f_2=\mathbf{1}_L$, $h_1=\mathbf{1}_{K \times L}$ and $h_2=\mathbf{1}_{K+L}$, and let $L_1 = {\rm Id}_{2n}$ and $L_2(x,y) := x+y$ for $x,y\in \R^n$. 
    Let also 
    $$
    \widetilde{\Gamma}_1 
    :=
    \begin{pmatrix}
        \Gamma_1 & \Gamma_1 \\
        \Gamma_1 & \Gamma_1
    \end{pmatrix}. 
    $$
    Then it holds that 
    $$
    f_1(x_1) f_2(x_2) \le h_1(x_1, x_2) h_2(L_2(x_1,x_2)), \;\;\; \forall x_1,x_2 \in \R^n.  
    $$
    Note also that $f_1, f_2, h_2 : \R^n \rightarrow \R_+$ and $h_1 : \R^{2n} \rightarrow \R_+$ are log-concave and satisfy 
    $$
    \int_{\R^n} x f_1(x) \, d\gamma_{\Sigma_1}(x)
    =
    \int_{\R^n} x f_2(x) \, d\gamma_{\Sigma_2}(x)
    =0
    $$
    by the assumption. 
    Consequently, Theorem \ref{t:FRBLMNT} confirms that
    $$
    \int_{\R^n} f_1\, d\gamma_{\Sigma_1} \int_{\R^n} f_2\, d\gamma_{\Sigma_2} 
    \le
    {\rm FR}_{LC}^{(\mathcal{G})}
    \int_{\R^n \times \R^n} h_1\, d\gamma_{\widetilde{\Gamma}_1}
    \int_{\R^n} h_2\, d\gamma_{\Gamma_2}, 
    $$
    where 
    $$
    {\rm FR}_{LC}^{(\mathcal{G})}
    =
    \sup
    \left\{
    \frac{{\rm det}\, ({\rm Id}_{2n} + \widetilde{\Gamma}_1 B)^{1/2} {\rm det}\, ({\rm Id}_n + \Gamma_2 C)^{1/2}}{{\rm det}\, ({\rm Id}_n + \Sigma_1 A_1)^{1/2}
    {\rm det}\, ({\rm Id}_n + \Sigma_2 A_2)^{1/2}}
    \, :\, \substack{\text{$A_1, A_2, C \in \mathcal{P}_n, B \in \mathcal{P}_{2n}$} \\ \text{ satisfying \eqref{e:TehCovMatAss}}}
    \right\} . 
    $$
    Since 
    $$
    {\rm det}\, ({\rm Id}_{2n} + \widetilde{\Gamma}_1 B)
    =
    {\rm det}\, ({\rm Id}_{n} + \Gamma_1 \overline{B}), 
    $$
    Lemma \ref{l:TehCovMat} verifies that 
    $
    {\rm FR}_{LC}^{(\mathcal{G})}
    \le
    1. 
    $
    Since
    $$
    \int_{\R^n} f_1\, d\gamma_{\Sigma_1}
    =
    \gamma_{\Sigma_1}(K), 
    \;\;\;
    \int_{\R^n} f_2\, d\gamma_{\Sigma_2}
    =
    \gamma_{\Sigma_2}(L), 
    \;\;\;
    \int_{\R^n} h_2\, d\gamma_{\Gamma_2}
    =
    \gamma_{\Gamma_2}(K+L) 
    $$
    and 
    $$
    \int_{\R^n \times \R^n} h_1\, d\gamma_{\widetilde{\Gamma}_1}
    =
    \int_{\R^n} h_1(y,y) \, d\gamma_{\Gamma_1}(y)
    =
    \gamma_{\Gamma_1}(K \cap L),  
    $$
    the proof is complete. 
\end{proof}

We now proceed to show Corollary \ref{c:TehGaussMix}. 
To this end, let us recall the four functions theorem \cite{AD} (see \cite[Lemma 9.3]{MNT} for the trivial extension from $\R^n$ to $(0,\infty)^n$). 
\begin{theorem}[Four functions theorem]\label{t:FourThm}
    Let $f_1, f_2, h_1, h_2$ be nonnegative and integrable functions on $(0, \infty)^n$ satisfying
    $$
    f_1(s) f_2(t) \le h_1(s \wedge t) h_2(s \vee t), \;\;\; \forall s, t \in (0, \infty)^n. 
    $$
    Then one has 
    $$
    \int_{(0, \infty)^n} f_1(s)\, ds \int_{(0, \infty)^n} f_2(t)\, dt
    \le
    \int_{(0, \infty)^n} h_1(r) \, dr \int_{(0, \infty)^n} h_2(r)\, dr.   
    $$
\end{theorem}

\begin{proof}[Proof of Corollary \ref{c:TehGaussMix}]
    For $s, t \in (0, \infty)^n$, put $\Sigma_1 = \Delta(s)$, $\Sigma_2= \Delta(t)$, $\Gamma_1 = \Delta(s \wedge t)$ and $\Gamma_2 = \Delta( s \vee t)$. 
    Then we see that 
    $$
    \Sigma_1 \le \Gamma_2, \;\;\;
    \Sigma_2 \le \Gamma_2, \;\;\;
    \Gamma_1^{-1} + \Gamma_2^{-1} = \Sigma_1^{-1} + \Sigma_2^{-1}, 
    \;\;\;
    {\rm det}\, \Sigma_1 {\rm det}\, \Sigma_2 = {\rm det}\, \Gamma_1 {\rm det}\, \Gamma_2.  
    $$
    Hence we may apply Theorem \ref{t:TehCov} to deduce that 
    \begin{equation}\label{e:PartsOfFour}
    \gamma_{\Delta(s)}(K) \gamma_{\Delta(t)}(L)
    \le
    \gamma_{\Delta(s \wedge t)}(K \cap L) \gamma_{\Delta (s \vee t)}(K + L) 
    \end{equation}
    whenever 
    $$
    \int_{K} x \, d\gamma_{\Delta(s)}(x) = \int_{L} x \, d\gamma_{\Delta(t)}(x) = 0, \;\;\; \forall s,t \in {\rm supp}(\nu).
    $$
    Let us now define the functions $f_1, f_2, h_1, h_2$ on $(0, \infty)^n$ by 
    $$
    f_1(s) := \gamma_{\Delta(s)}(K) f_\nu(s), 
    \;\;\;
    f_2(t) := \gamma_{\Delta(t)}(L) f_\nu(t), 
    $$
    and 
    $$
    h_1(r) := \gamma_{\Delta(r)}(K \cap L) f_\nu(r), 
    \;\;\;
    h_2(r) := \gamma_{\Delta(r)}(K+L) f_\nu(r). 
    $$
    By \eqref{e:PartsOfFour} and log-supermodularity of $f_\nu$, we have 
    $$
    f_1(s) f_2(t) \le h_1(s \wedge t) h_2 (s \vee t), \;\;\; \forall s,t \in (0, \infty)^n. 
    $$
    Therefore Theorem \ref{t:FourThm} yields 
    $$
    \int_{(0, \infty)^n} f_1(s)\, ds \int_{(0, \infty)^n} f_2(t) \, dt
    \le
    \int_{(0, \infty)^n} h_1(r) \, dr \int_{(0, \infty)^n} h_2(r)\, dr,    
    $$
    establishing our desired inequality and concluding the proof. 
\end{proof}

\end{document}